\documentclass{article}
\usepackage[utf8]{inputenc}

\title{Congruence of Linear Symplectic Forms by the Symplectic Group}
\author{Luchen Shi, Sunay Joshi, Ritwick Bhargava}
\date{December 2022}

\usepackage{graphicx}
\usepackage{fullpage}
\usepackage{amssymb}
\usepackage{amsmath}
\usepackage{amsthm}
\usepackage{xcolor}
\usepackage{comment}
\usepackage{appendix}

\providecommand{\keywords}[1]
{
  \small	
  \textbf{Keywords:} #1
}

\newtheorem{theorem}{Theorem}[section]
\newtheorem{corollary}{Corollary}[theorem]
\newtheorem{lemma}[theorem]{Lemma}
\newtheorem{question}{Question}
\newtheorem{Proposition}[theorem]{Proposition}
\newtheorem{conjecture}{Conjecture}

\theoremstyle{definition}
\newtheorem{definition}{Definition}[section]

\theoremstyle{definition}
\newtheorem*{example}{Example}

\DeclareMathOperator{\Symp}{Symp}
\DeclareMathOperator{\Diff}{Diff}
\DeclareMathOperator{\Sp}{Sp}

\begin{document}

\newcommand{\ssep}{\mid}
\newcommand*{\R}{\mathbb{R}}
\newcommand{\omegastd}{\omega_{\mathrm{std}}}
\newcommand{\omegalstd}{\omega_{\mathrm{lstd}}}
\maketitle

\begin{abstract}

  This paper concerns the action of linear symplectomorphisms on linear symplectic forms by conjugation in even dimensions. We prove that pfaffian and $-\frac{1}{2}\operatorname{tr}(JA)$ (sum function) of $A$ are invariants on the action. We use these invariants to provide a complete description of the orbit space in dimension four. In addition, we investigate the geometric shapes of the individual orbits in dimension four.
  
  In symplectic geometry, our classification result in dimension four provides a necessary condition for two symplectic forms on $\R^{4}$ to be intertwined by symplectomorphisms of the standard symplectic form. This stands in contrast to the lack of local invariants under diffeomorphisms. Furthermore, we determine global invariants of a class of symplectic forms, and we study an extension of a corollary of the Curry-Pelayo-Tang Stability Theorem.
  
  Lastly, we extend our results and investigate the action of linear symplectomorphisms on linear symplectic forms in dimension $2n$. We determine $n$ invariants of linear symplectic forms under this action, namely, $s_k(A)$ we defined as $\sigma_k(A)$ which is the coefficient of term $t^k$ in the polynomial expansion of pfaffian of $tJ+A$.

\end{abstract}


\section{Introduction}

\subsection{Motivation}
Symplectic structures on $\R^4$ have been studied for decades.There have been many important results on the diffeomorphism classes of symplectic structures (the equivalence classes of symplectic structures under diffeomorphisms): see references \cite{Curry2018} \cite{Darboux1882} \cite{ElGr} \cite{Gr85} \cite{McD88}.
In this paper, we restrict the use of diffeomorphisms to a narrower scope: the set of symplectomorphisms of $(\R^4, \omegastd)$.
The symplectomorphism group $\Symp(\R^4, \omegastd)$ consists of diffeomorphisms of $\R^4$ which preserve the standard symplectic form $\omegastd$. The structure of $\Symp(\R^4, \omegastd)$ has been discussed in references \cite{McD04} \cite{Ld01} \cite{Salamon}.
The elements of this group do not preserve symplectic forms other than the scalar multiples of $\omegastd$ in general, and hence, we would like to know how a general symplectic form on $\R^4$ can be acted by $\Symp(\R^4, \omegastd)$.

In this paper, we settle the linear version of this general question.
The answer to the linear problem yields a necessary condition for a symplectic form on $\R^4$ to be transformed to another by $\Symp(\R^4, \omegastd)$, in the sense that it provides both local invariants (the Pfaffian and the sum function) and global invariants under symplectomorphisms. Note that this stands in contrast to the lack of local invariants under diffeomorphisms. Furthermore, based on our proof in four dimensions, we prove that the Pfaffian and sum function remain invariant in every even dimension. These results may help in understanding the structure of the symplectomorphism group $\Symp(\R^{2n}, \omegastd)$.

We also investigate the geometry of the orbits of the linear group action in dimension four.

\subsection{Goals}

(Note: Throughout this paper, all vectors and matrices are written with respect to the standard basis.)

Let
\[ S(2n,\R)=\{ A\in M(2n,\R) \ssep A^T = -A, \det{A} \neq 0\} \]
denote the set of all linear symplectic forms on $\R^{2n}$, and let the standard linear symplectic form on $\R^{2n}$ have matrix representation
\[
J =
\begin{bmatrix}
J_0 & O & \dots & O \\
O & J_0 & \ddots & \vdots \\
\vdots & \ddots & \ddots & O \\
O & \dots & O & J_0
\end{bmatrix},
\]

where 
$J_0 = 
\begin{bmatrix}
0 & 1 \\
-1 & 0
\end{bmatrix}$. Denote the set
\[ \Sp(2n) = \{ P\in M(2n,\R) \ssep P^T J P = J\} \]as the group of linear symplectomorphisms on $\R^{2n}$.Consider the group action $\rho : \Sp(2n) \times S(2n, \R) \rightarrow S(2n, \R)$ given by $\rho(P, A) = P^T A P$ for $P\in \Sp(2n)$ and $A\in S(2n,\R)$. In the case of $\R^4$, we have the matrix representation $J$ as below according to the above definitions:
\[
J = 
\begin{bmatrix}
0 & 1 & 0 & 0 \\
-1 & 0 & 0 & 0 \\
0 & 0 & 0 & 1 \\
0 & 0 & -1 & 0
\end{bmatrix},
\]

\begin{question}
  What is the orbit space $S(4,\R) / \Sp(4)$?
\end{question}

\begin{question}
  Which quantities are invariant under the group action $\rho$ in $\R^{2n}$?
\end{question}

\begin{question}
  What is the shape of each orbit, when considered as a subset of $\R^6$?
\end{question}

\begin{question}
  How does the linear result apply to the nonlinear case?
\end{question}

\begin{question}
  What is the orbit space classification in higher dimensions? 
\end{question}

The primary goal of this paper is to answer Question 1 and Question 2. At the end of the paper, we partially address Questions 3, 4, and 5.

The structure of this paper is as follows: in Sections 3 and 4, we present invariants of the group action $\rho$ in even dimensions as well as the Orbit Space Classification of $S(4,\R)/\Sp(4)$. We then explore the geometric properties of the individual orbits of $S(4,\R)/\Sp(4)$ in Section 5. Finally, we employ the language of symplectic forms to study the nonlinear case in $\R^4$ in Section 6 and to establish higher dimensional linear results in Section 7.

\section{Main Results}

\subsection{Linear Case in $\R^4$}

\begin{definition}[Pfaffian of $4\times 4$ skew-symmetric matrix]

The \textit{Pfaffian} of the $4\times 4$ skew-symmetric matrix
\[
A = 
\begin{bmatrix}
0 & a & b & c \\
-a & 0 & d & e \\
-b & -d & 0 & f \\
-c & -e & -f & 0
\end{bmatrix},
\]

is the quantity $\operatorname{Pf}(A) = af-be+cd$.

\end{definition}

It is well-known that for any $A \in S(4, \R)$, $\det(A) = \operatorname{Pf}(A)^2$ \cite[Proposition 2.2]{PfaffianProof}.

\begin{definition}[Sum function of $4\times 4$ skew-symmetric matrix]

The \textit{sum function} of the $4\times 4$ skew-symmetric matrix
\[
A = 
\begin{bmatrix}
0 & a & b & c \\
-a & 0 & d & e \\
-b & -d & 0 & f \\
-c & -e & -f & 0
\end{bmatrix},
\]

is the quantity $\operatorname{s}(A) = a+f$.

\end{definition}

\begin{theorem}[Invariants in $\R^4$]

The Pfaffian and sum function are invariant under the group action $\rho$.

\end{theorem}

\begin{theorem}[Orbit Space Classification in $\R^4$]

The orbit space $S(4,\R) / \Sp(4)$ consists of the following families of distinct orbits:

\begin{itemize}

    \item For all $p > 0$, 
    \[\mathcal{J}_p^{+} = \{\sqrt{p}J\}\]
    
    \item For all $p > 0$, 
    \[\mathcal{J}_p^{-} = \{-\sqrt{p}J\}\]
    
    \item For all $p > 0$, $q \in \R$, 
    \[\mathcal{A}_{p,q}^{+} = \{ A \in S(4,\R) \mid A \not = \pm\sqrt{p} J, \operatorname{Pf}(A) = p, \operatorname{s}(A) = q\}\]
    
    \item For all $p < 0$, $q \in \R$, 
    \[\mathcal{A}_{p,q}^{-} = \{ A \in S(4,\R) \mid \operatorname{Pf}(A) = p, \operatorname{s}(A) = q\}\]
    
\end{itemize}

\end{theorem}

We also present the shapes of two families of orbits of $S(4, \R) / \Sp(4)$.

\begin{theorem}[Shapes of Orbits in $\R^4$]

Let $p \in \R^*$, $q \in \R$, and define $\Delta = \frac{q^2}{4}-p$.

Then:

\begin{itemize}
    
    \item If $\Delta < 0$, $\mathcal{A}_{p,q}$ is homeomorphic to $(\R^2 \backslash \{0,0\}) \times \R^2$;
    
    \item If $\Delta > 0$, $\mathcal{A}_{p,q}$ is homeomorphic to $\mathbb{S}^2 \times \R^2$.

\end{itemize}

\end{theorem}

\subsection{Nonlinear Case in $\R^4$ and $\R^{2n}$}

\begin{definition}[Linear symplectic form]
A \textit{linear symplectic form} $\omega$ on $\R^{2n}$ is a map $\omega: \R^{2n} \times \R^{2n} \rightarrow \R$ which satisfies

\begin{itemize}

    \item $\omega$ is bilinear;

    \item $\omega$ is skew-symmetric: if $v, w\in \R^{2n}$, then $\omega(v,w)=-\omega(w,v)$; and

    \item $\omega$ is non-degenerate: if $\omega(v,w)=0$ for all $w \in \R^{2n}$, then $v=0$.

\end{itemize}
\end{definition}

The pair $(\R^{2n}, \omega)$ refers to the vector space $\R^{2n}$ equipped with the linear symplectic form $\omega$. We call $(\R^{2n}, \omega)$ a symplectic vector space.

Any linear symplectic form $\omega$ can be represented by the action of a characteristic matrix $\Omega$: if $\{e_1, \ldots, e_{2n}\}$ is the standard basis for $\R^{2n}$, and if $\Omega$ is the $2n \times 2n$ matrix such that $\Omega_{i,j} = \omega(e_i, e_j)$ for all $1 \le i, j \le 2n$, then $\omega(v,w) = v^T \Omega w$ for all $v, w \in \R^{2n}$ \cite[Proposition 2.8.14]{FirstSteps}. We shall refer to $\Omega$ as the matrix representation of $\omega$.

\begin{definition}[Standard linear symplectic form]

Let $J$ denote the $2n\times 2n$ matrix with block matrix representation
\[
J =
\begin{bmatrix}
J_0 & O & \dots & O \\
O & J_0 & \ddots & \vdots \\
\vdots & \ddots & \ddots & O \\
O & \dots & O & J_0
\end{bmatrix},
\]

where 
$J_0 = 
\begin{bmatrix}
0 & 1 \\
-1 & 0
\end{bmatrix}$. The \textit{standard linear symplectic form} on $\R^{2n}$ is the linear symplectic form $\omegalstd$ defined as $\omegalstd(v,w) = v^T J w$ for all $v, w \in \R^{2n}$.

\end{definition}

\begin{definition}[Linear symplectomorphism]

A linear transformation $T : \R^{2n} \rightarrow \R^{2n}$ is a \textit{linear symplectomorphism of $(\R^{2n}, \omegalstd)$} if $\omegalstd(T(v), T(w)) = \omegalstd(v, w)$ for all $v, w \in \R^{2n}$. 

\end{definition}

Note that a linear transformation is a linear symplectomorphism if and only if its matrix representation $P$ satisfies $P^T J P = J$ \cite[Theorem 2.10.20]{FirstSteps}.

We now turn to the definition of nonlinear symplectic forms \cite[Definition 1.5]{DefinitionsOfNonlinearForms}.

\begin{definition}[$2$-form]

A \textit{$2$-form on $\R^{2n}$} is an expression of the form 
\[\sum_{1\le i < j\le 2n} \omega_{i,j}\mathrm{d} x_i \wedge \mathrm{d} x_j,\]

where $\omega_{i,j} : \R^{2n} \rightarrow \R$ for all $1\le i < j \le 2n$.

\end{definition}

A $2$-form $\omega$ on $\R^{2n}$ is an alternating bilinear form $\omega(x)$ at every point $x \in \R^{2n}$. Furthermore, the alternating bilinear form $\omega(x)$ can be represented by the action of a characteristic matrix $\Omega(x)$: if $\Omega(x)$ is the skew-symmetric matrix such that $\Omega(x)_{i,j} = \omega_{i,j}$ for all $1\le i < j \le 2n$, then $\omega(x)(v,w) = v^T \Omega(x) w$ for all $v, w \in \R^{2n}$. We shall refer to $\Omega(x)$ as the matrix representation of $\omega(x)$.

\begin{definition}[Symplectic form] A \textit{symplectic form on $\R^{2n}$} is a $2$-form
\[\omega = \sum_{1\le i < j\le 2n} \omega_{i,j}\mathrm{d} x_i \wedge \mathrm{d} x_j\]

satisfying the following conditions:

\begin{itemize}
    
    \item $\omega$ is smooth: $\omega_{i,j}$ is smooth for all $1\le i < j\le 2n$;
    
    \item $\omega$ is closed: for all $1\le i < j < k\le 2n$,
    \[\frac{\partial{\omega_{j,k}}}{\partial{x_i}} - \frac{\partial{\omega_{i,k}}}{\partial{x_j}} + \frac{\partial{\omega_{i,j}}}{\partial{x_k}} = 0;\]
    and
    
    \item $\omega$ is nondegenerate: $\det(\Omega(x)) \neq 0$ for all $x \in \R^{2n}$, where $\Omega$ is the matrix representation of $\omega$.
    
\end{itemize}

\end{definition}

As above, the symplectic vector space $(\R^{2n}, \omega)$ represents $\R^{2n}$ equipped with the symplectic form $\omega$. 

Note that if $\omega$ is a symplectic form, then at each point $x \in \R^{2n}$, the alternating bilinear form $\omega(x)$ is a linear symplectic form.

\begin{definition}[Standard symplectic form]

The \textit{standard symplectic form on $\R^{2n}$} is the symplectic form

\[\omegastd = \sum_{i=1}^{n} \mathrm{d} x_{2i-1} \wedge \mathrm{d} x_{2i}.\]

\end{definition}

Symplectomorphisms may be defined in terms of the pullback \cite[Definition 7.3.1]{FirstSteps}.

\begin{definition}[Pullback of a $2$-form]

Let $\varphi : \R^{2n} \rightarrow \R^{2n}$ be a smooth map, and let $\omega$ be a $2$-form. The \textit{pullback of $\omega$ by $\varphi$} is the $2$-form defined as
\[(\varphi^* \omega)(x)(v, w) =
\omega(\varphi(x))(d\varphi_{x}(v), d\varphi_{x}(w))\]

for all $v, w \in \R^{2n}$. (Here, $d\varphi_{x}$ denotes the linearization of $\varphi$ at $x$.)

\end{definition}

\begin{definition}[Symplectomorphism]
A \textit{symplectomorphism $\varphi$ of $(\R^{2n}, \omegastd)$} is a diffeomorphism of $\R^{2n}$ satisfying $\varphi^* \omegastd = \omegastd$.

\end{definition}

We also introduce the notion of two symplectic forms being $\omegastd$-symplectomorphic.

\begin{definition}
The symplectic forms $\omega_1$ and $\omega_2$ on $\R^{2n}$ are \textit{$\omegastd$-symplectomorphic} if there exists $\varphi \in \Symp(\R^{2n}, \omegastd)$ such that $\varphi^* \omega_1 = \omega_2$.

\end{definition}

With the above definitions in place, we may present our results for the nonlinear case in $\R^{2n}$. Let $\R^{2n}$ have coordinates $(x_1, y_1, x_2, y_2,...,x_{2n},y_{2n})$. First, we note that the invariants of Theorem 2.1 provide a necessary condition for two nonlinear symplectic forms on $\R^{2n}$ to be $\omegastd$-symplectomorphic.

\begin{theorem}

Let $\omega_1$ and $\omega_2$ be symplectic forms on $\R^{2n}$. Then $\omega_1$ and $\omega_2$ are $\omegastd$-symplectomorphic only if there exists $\varphi \in \Symp(\R^{2n}, \omegastd)$ such that for every point $x \in \R^{2n}$, the Pfaffians and sum functions of $\Omega_1(\varphi(x))$ and $\Omega_2(x)$ match. (Here, $\Omega_1(\varphi(x))$ and $\Omega_2(x)$ are the matrix representations of $\omega_1(\varphi(x))$ and $\omega_2(x)$, respectively.)

\end{theorem}

\begin{theorem}
     Consider symplectic forms 
     \begin{equation*}
         \omega_1=\sum_{i=1}^n f_i(x_i,y_i)dx_i\wedge dy_i
     \end{equation*}
and 
\begin{equation*}
         \omega_2=\sum_{i=1}^n g_i(x_i,y_i)dx_i\wedge dy_i
\end{equation*}
If $\phi \in \Symp(\mathbb{R}^{2n}, \omega_{std})$, such that $\phi^*\omega_1=\omega_2$, $\phi=(\phi_1,\phi_2,...,\phi_n)$
Then the set $\{f_i(\phi^{2i-1},\phi^{2i})\}$ and $\{g_j(\phi^{2j-1},\phi^{2j})\}$ for some $i,j$ such that $1<i,j<n$ at every point in $\mathbb{R}^{2n}$. 
\end{theorem}

In particular, for the class of symplectic forms
\[\omega_1=\sum_{i=1}^n f_i(x_i,y_i)dx_i\wedge dx_i\]

where $f_i,1\le i \le n$ are nowhere vanishing, we derive multiple global invariants.

\begin{theorem}[Global Invariants]

Consider the symplectic form
\[\omega_1=\sum_{i=1}^n f_i(x_i,y_i)dx_i\wedge dx_i\]

where $f_i,1\le i \le n$ are nowhere vanishing. For $x \in \R^{2n}$, define 
\[m_{\omega}(x) = \min\{f_i(x_i,y_i), 1\le i \le n\}\]
and 
\[M_{\omega}(x) = \max\{f_i(x_i,y_i), 1\le i \le n\}.\]

Then the following four quantities are global invariants of $\omega$:

\begin{itemize}
    \item $\inf \{m_{\omega}(x) \mid x \in \mathbb{R}^{2n}\}$,
    
    \item $\sup\{m_{\omega}(x) \mid x \in \mathbb{R}^{2n}\}$,
    
    \item $\inf\{M_{\omega}(x) \mid x \in \mathbb{R}^{2n}\}$,
    
    \item $\sup\{M_{\omega}(x) \mid x \in \mathbb{R}^{2n}\}$.
\end{itemize}

In other words, if $\omega_1$ is $\omegastd$-symplectomorphic to the symplectic form
\[\omega_2 =\sum_{i=1}^n g_i(x_i,y_i)dx_i\wedge dy_i\]

where $g_i,1\le i \le n$ are nowhere vanishing, then the four global invariants of $\omega_1$ and $\omega_2$ must match.

\end{theorem}

Finally, we consider the symplectic forms 
\[\omega_t = f_1(t, x_1, y_1) \mathrm{d} x_1 \wedge \mathrm{d} y_1 + f_2(t, x_2, y_2) \mathrm{d} x_2 \wedge \mathrm{d} y_2, t \in [0, 1],\]

where the functions $f_i$ are bounded away from zero with bounded time derivative. The Curry-Pelayo-Tang Stability Theorem implies that there exists a smooth path of $\varphi_t \in \Diff(\R^4)$, $t \in [0, 1]$, such that $\varphi_t^* \omega_t = \omega_0$ \cite[Example 4.1]{Curry2018}.

We present an infinite collection of smooth families of symplectic forms which are diffeomorphic but \textit{not} $\omegastd$-symplectomorphic. For instance, the family of symplectic forms 
\[\omega_t = (x_1^2+y_1^2+t+1)\mathrm{d} x_1 \wedge \mathrm{d} y_1 + (x_2^2+y_2^2+t+1)\mathrm{d} x_2 \wedge \mathrm{d} y_2, t \in [0,1]\]

are diffeomorphic but not $\omegastd$-symplectomorphic.

\subsection{Higher Dimensional Linear Case}

Lastly, we study a generalization of our original linear classification problem to $\R^{2n}$.

Let
\[ S(2n,\R)=\{ A\in M(2n,\R) \ssep A^T = -A, \det{A} \neq 0\} \]
denote the set of all linear symplectic forms on $\R^{2n}$, and let
\[ \Sp(2n) = \{ P\in M(2n,\R) \ssep P^T J P = J\} \]
denote the group of linear symplectomorphisms on $\R^{2n}$.

As above, we consider the group action $\rho : \Sp(2n) \times S(2n, \R) \rightarrow S(2n, \R)$ given by $\rho(P,A) = P^T A P$ for $P\in \Sp(2n)$ and for $A\in S(2n,\R)$.

We present a general Equivalence Lemma for $\R^{2n}$. We also demonstrate that the generalized Pfaffian and generalized sum function remain invariants of the group action $\rho$ \cite[Definition 2.1]{PfaffianProof}.

\begin{definition}[Pfaffian of a $2n \times 2n$ matrix]

The \textit{Pfaffian} of a $2n \times 2n$ matrix $A$ is the quantity
\[\operatorname{Pf}(A) = \sum_{\sigma} \text{sgn}(\sigma) \prod_{i=1}^{n} A_{\sigma(2i-1), \sigma(2i)},\]

where $\text{sgn}(\sigma)$ denotes the signature of $\sigma$ and the summation is taken over all permutations $\sigma$ of $\{1, 2, \ldots, 2n\}$. 
(Here, $A_{i, j}$ denotes the $(i, j)$-entry of $A$.)

\end{definition}

Note that the identity $\det(A) = \operatorname{Pf}(A)^2$ holds in general \cite[Proposition 2.2]{PfaffianProof}.

\begin{definition}[Sum function of $2n \times 2n$ matrix]

The \textit{sum function} of a matrix $A \in S(2n,\R)$ is the quantity \[\operatorname{s}(A) = \sum_{i=1}^{n} A_{2i-1, 2i}.\]

\end{definition}

\begin{theorem}[Invariants in $\R^{2n}$]

The Pfaffian of a matrix in $S(2n,\R)$ is invariant under the group action $\rho$. In addition, the sum function of a matrix in $S(2n,\R)$ is invariant under the group action $\rho$.

\end{theorem} 

\begin{theorem}[Invariants in $\R^{2n}$]
Let $A \in S(2n,\mathbb{R})$, and let $0\le k\le n-1$. For a subset $S \subset \{1,2,...,n\}$ with $|S| = k$, define the $(2n-2k) \times (2n-2k)$ matrix $A_{S}$ as the submatrix of $A$ formed by removing rows and columns $2i-1$ and rows and columns $2i$ for all $i \in S$.Let $s_k(A)$ denote the sum
\[
\sum_{\substack{S \subseteq \{1, 2, \ldots, n\} \\ |S|=k}} \operatorname{Pf}(A_{S}).
\]
Then for any $P \in \Sp(2n)$, $s_k(A) = s_k(P^T A P)$. In other words, $s_k$ is invariant under the group action $\rho$.
\end{theorem}

\section{Proof of pfaffian and sum function in $\R^4$}

In this section, we prove Theorem 2.1 and 2.7 together. We begin with the invariance of the Pfaffian under $\rho$. We will need two well-known lemmas \cite[Proposition 2.3]{PfaffianProof} \cite[Corollary 7.3.5]{FirstSteps}.

\begin{lemma}

Let $A \in S(2n,\R)$ and let $P \in M(2n,\R)$. Then $\operatorname{Pf}(P^T A P) = \operatorname{Pf}(A) \det(P)$.

\end{lemma}

\begin{lemma}

If $P \in \Sp(2n)$, then $\det(P) = 1$.

\end{lemma}

As a corollary of Lemmas 3.1 and 3.2, we have the invariance of the Pfaffian.

\begin{lemma}[Pfaffian invariant in $\R^4$]

Let $A \in S(2n,\R)$ and let $P \in \Sp(2n)$. Then $\operatorname{Pf}(P^T A P) = \operatorname{Pf}(A)$. In other words, the Pfaffian is invariant under the group action $\rho$.

\end{lemma}

\begin{proof}

By Lemma 3.1, $\operatorname{Pf}(P^T A P) = \operatorname{Pf}(A)\det(P)$. Furthermore, by Lemma 3.2, $\det{P} = 1$. The result follows.

\end{proof}

Next, we shall prove the invariance of the sum function under $\rho$. Again, we will need two lemmas.

\begin{lemma}

If $P \in \Sp(2n)$, then $P^T \in \Sp(2n)$.

\end{lemma}

\begin{proof}

By definition, if $P \in \Sp(2n)$, then $P^T J P = J$.
Hence, 
\[J P^T J P = J^2 = -I\]
\[\implies J P^T = -(JP)^{-1} = P^{-1}J\]
\[\implies (P^T)^T J P^T = P J P^T = J,\]
as desired.

\end{proof}

\begin{lemma}

Let $P \in \Sp(2n)$. Let the column vectors of $P$ be $c_1, c_2, \ldots, c_{2n}$, and let the row vectors of $P$ be $r_1, r_2, \ldots, r_{2n}$. Then the sets $\{c_1, c_2, \ldots, c_{2n}\}$ and $\{r_1, r_2, \ldots, r_{2n}\}$ form symplectic bases for $(\R^{2n}, \omegalstd)$.

\end{lemma}

\begin{proof}

Collapse the column vectors $c_1, c_2, \ldots, c_{2n}$ of $P$ and write
\[
P = 
\begin{bmatrix}
c_1 & c_2 & \dots & c_{2n}
\end{bmatrix}.
\] 

By definition, $P^T J P = J$. Therefore,
\[
\begin{bmatrix}
c_1^T \\
c_2^T \\
\vdots \\
c_{2n}^T
\end{bmatrix}
J
\begin{bmatrix}
c_1 & c_2 & \dots & c_{2n}
\end{bmatrix}
=
\begin{bmatrix}
c_1^T \\
c_2^T \\
\vdots \\
c_{2n}^T
\end{bmatrix}
\begin{bmatrix}
Jc_1 & Jc_2 & \dots & Jc_{2n}
\end{bmatrix}
\]

\[
=
\begin{bmatrix}
c_1^T J c_1 & c_1^T J c_2 & \dots & c_1^T J c_{2n} \\
c_2^T J c_1 & c_2^T J c_2 & \dots & \vdots \\
\vdots & \vdots & \ddots & c_{2n-1}^T J c_{2n} \\
c_{2n}^T J c_1 & \dots & c_{2n}^T J c_{2n-1} & c_{2n}^T J c_{2n}
\end{bmatrix}
\]

\[
=
\begin{bmatrix}
\omegalstd(c_1,c_1) & \omegalstd(c_1,c_2) & \dots & \omegalstd(c_1,c_{2n}) \\
\omegalstd(c_2,c_1) & \omegalstd(c_2,c_2) & \dots & \vdots \\
\vdots & \vdots & \ddots & \omegalstd(c_{2n-1},c_{2n}) \\
\omegalstd(c_{2n},c_1) & \dots & \omegalstd(c_{2n},c_{2n-1}) & \omegalstd(c_{2n},c_{2n})
\end{bmatrix}
\]

\[
= J
=
\begin{bmatrix}
0 & 1 & \dots & \dots & 0 \\
-1 & 0 & \ddots & \ddots & \vdots \\
\vdots & \ddots & \ddots & \ddots & \vdots \\
\vdots & \ddots & \ddots & 0 & 1 \\
0 & \dots & \dots & -1 & 0
\end{bmatrix}.
\]

Equating corresponding entries, we find that $\{c_1, c_2, \ldots, c_{2n}\}$ forms a symplectic basis for $(\R^{2n}, \omegalstd)$. To see that the same holds for the row vectors of $P$, note that by Lemma 7.6, $P^T \in \Sp(2n)$. The result follows.

\end{proof}

We are now ready to prove the invariance of the sum function.

\begin{lemma}[Sum function invariant in $\R^{2n}$]

Let $A \in S(2n,\R)$ and let $P \in \Sp(2n)$. Then $\operatorname{s}(P^T A P) = \operatorname{s}(A)$. In other words, the sum function is invariant under the group action $\rho$.

\end{lemma}

\begin{proof}

Let the matrix representation of $P$ be
\[
\begin{bmatrix}
a_{1,1} & a_{1,2} & \dots & \dots & a_{1,2n} \\
a_{2,1} & a_{2,2} & \ddots & \ddots & \vdots \\
\vdots & \ddots & \ddots & \ddots & \vdots \\
\vdots & \ddots & \ddots & a_{2n-1,2n-1} & a_{2n-1,2n} \\
a_{2n,1} & \dots & \dots & a_{2n,2n-1} & a_{2n,2n}
\end{bmatrix}.
\]

Let $c_1, c_2, \ldots, c_{2n}$ denote the column vectors of $P$, and let $r_1, r_2, \ldots r_{2n}$ denote the row vectors of $P$. Let $\omega$ denote the linear symplectic form with matrix representation $A$.

Then the matrix representation of $P^T A P$ is
\[
\begin{bmatrix}
\omega(c_1,c_1) & \omega(c_1,c_2) & \dots & \omega(c_1,c_{2n}) \\
\omega(c_2,c_1) & \omega(c_2,c_2) & \ddots & \vdots \\
\vdots & \ddots & \ddots & \omega(c_{2n-1},c_{2n}) \\
\omega(c_{2n},c_1) & \dots & \omega(c_{2n},c_{2n-1}) & \omega(c_{2n},c_{2n})
\end{bmatrix}.
\]

Furthermore, since the matrix representation of $A$ can be rewritten as
\[
\begin{bmatrix}
\omega(e_1,e_1) & \omega(e_1,e_2) & \dots & \omega(e_1,e_{2n}) \\
\omega(e_2,e_1) & \omega(e_2,e_2) & \ddots & \vdots \\
\vdots & \ddots & \ddots & \omega(e_{2n-1},e_{2n}) \\
\omega(e_{2n},e_1) & \dots & \omega(e_{2n},e_{2n-1}) & \omega(e_{2n},e_{2n})
\end{bmatrix},
\]

it suffices to show that $\operatorname{s}(P^T A P) = \sum_{k=1}^{n} \omega(c_{2k-1}, c_{2k})$ is equal to $\operatorname{s}(A) = \sum_{k=1}^{n} \omega(e_{2k-1}, e_{2k})$. To see this, note that
\[c_k = \sum_{j=1}^{2n} a_{j,k}e_j\]

for all $1\le k\le 2n$. Hence, by the bilinearity of $\omega$,
\[\sum_{k=1}^{n} \omega(c_{2k-1}, c_{2k}) = \sum_{1\le i < j\le 2n} \left(\sum_{k=1}^{n} (a_{i, 2k-1}a_{j, 2k} - a_{i, 2k}a_{j, 2k-1})\right)\omega(e_i, e_j)\]

\[= \sum_{1\le i < j\le 2n} \omegalstd(r_i, r_j) \omega(e_i, e_j).\]

Yet, by Lemma 7.7, $\{r_1, r_2, \ldots, r_{2n}\}$ is a symplectic basis for $(\R^{2n}, \omegalstd)$. Therefore, the above sum reduces to $\sum_{k=1}^{2n} \omega(e_{2k-1}, e_{2k}) = \operatorname{s}(A)$, as desired.

\end{proof}

Theorem 2.1 and 2.7 follows.

\section{Proof of Orbit Space Classification in $\R^4$}

\subsection{Equivalence Lemma for $\R^4$}

In order to establish the Orbit Space Classification of Theorem 2.2, we will develop a method to classify equivalent skew-symmetric matrices based on their so-called symplectic bases.

\begin{definition}[Symplectic basis in $\R^4$]

Let $A \in S(4,\R)$ and let $\mathcal{B} = \{v_1, v_2, v_3, v_4\}$ be a basis for $\R^4$.  The basis $\mathcal{B}$ is said to be a \textit{symplectic basis for $A$} if the following equalities hold:
\[v_1^T A v_2 = 1\]
\[v_1^T A v_3 = 0\]
\[v_1^T A v_4 = 0\]
\[v_2^T A v_3 = 0\]
\[v_2^T A v_4 = 0\]
\[v_3^T A v_4 = 1.\]

\end{definition}

The following theorem is well-known consequence of the Modified Gram-Schmidt Algorithm \cite[Theorem 2.10.4]{FirstSteps}.

\begin{theorem}

If $A \in S(4,\R)$, then there exists a basis $\mathcal{B}$ for $\R^4$ which is a symplectic basis for $A$.

\end{theorem}

\begin{lemma}

Let $\{v_1, v_2, v_3, v_4\}$ and $\{w_1, w_2, w_3, w_4\}$ be two bases for $\R^4$. Define the linear transformation $P : \R^4 \rightarrow \R^4$ such that $Pv_i = w_i$ for $1\le i\le 4$. Then $P^T J P = J$ if and only if $v_i^T J v_j = w_i^T J w_j$ for all $1\le i, j \le 4$.

\end{lemma}

\begin{proof}

($\rightarrow$) If $P^T J P = J$, then 
\[w_i^T J w_j = (Pv_i)^T J (Pv_j) = v_i^T (P^T J P) v_j = v_i^T J v_j,\]
as desired.

($\leftarrow$) We may rewrite the given condition as \[v_i^T J v_j = (Pv_i)^T J (Pv_j) = v_i^T (P^T J P) v_j\] for all $1\le i, j\le 4$. As $\{v_1, v_2, v_3, v_4\}$ forms a basis for $\R^4$, it follows that for all $x, y \in \R^4$, $x^T J y = x^T (P^T J P) y$. In particular, for the standard basis vectors $x = e_i$ and $y = e_j$, we have $J_{i,j} = e_i^T J e_j = e_i^T (P^T J P) e_j = (P^T J P)_{i,j}$. It follows that all of the entries of $J$ and $P^T J P$ are equal, so that $J = P^T J P$. The result follows.

\end{proof}

The following definition will allow us to present our primary method of matrix classification, the Equivalence Lemma.

\begin{definition}[Basis-values for $\R^4$]

Given a symplectic basis $\mathcal{B} = \{v_1, v_2, v_3, v_4\}$ for the matrix $A \in S(4,\R)$, the \textit{set of basis-values of $\mathcal{B}$} is the ordered $6$-tuple of real numbers \[(v_1^T J v_2, v_1^T J v_3, v_1^T J v_4, v_2^T J v_3, v_2^T J v_4, v_3^T J v_4).\]

\end{definition}

At last, we may introduce the following corollary of Lemma 4.2.

\begin{corollary}[Equivalence Lemma in $\R^4$]

Let $A, B \in S(4,\R)$. Let $\mathcal{B}_1 = \{v_1, v_2, v_3, v_4\}$ be a symplectic basis for $A$ and let $\mathcal{B}_2 = \{w_1, w_2, w_3, w_4\}$ be a symplectic basis for $B$. If the set of basis-values of $\mathcal{B}_1$ is equal to the set of basis-values of $\mathcal{B}_2$, then $A$ and $B$ are equivalent under the group action $\rho$.

\end{corollary}

\begin{proof}

Let $P : \R^4 \rightarrow \R^4$ be the linear transformation such that $Pv_i = w_i$ for $1\le i\le 4$. As the set of basis-values of $\mathcal{B}_1$ is equal to the set of basis-values for $\mathcal{B}_2$, by Lemma 4.2, it follows that $P^T J P = J$.

As $\mathcal{B}_1$ and $\mathcal{B}_2$ are symplectic bases for $A$ and $B$, respectively, $v_i^T A v_j = w_i^T B w_j$ for all $1\le i,j \le 4$. Hence, $v_i^T A v_j = (Pv_i)^T B (Pv_j)^T$ for all $1\le i,j \le 4$, which implies that $v_i^T A v_j = v_i^T (P^T B P) v_j$ for all $1\le i,j \le 4$. 

As $\{v_1, v_2, v_3, v_4\}$ is a basis for $\R^4$, it follows that $x^T A y = x^T (P^T B P) y$ for all $x, y \in \R^4$. In particular, for the standard basis vectors $x = e_i$ and $y = e_j$, we have $A_{i,j} = e_i^T A e_j = e_i^T (P^T B P) e_j = (P^T B P)_{i,j}$. It follows that all of the entries of $A$ and $P^T B P$ are equal, so that $A = P^T B P$. The result follows.

\end{proof}

The Equivalence Lemma will be employed extensively in the following section. A few remarks on its use are in order. Recall that the set of basis-values of a basis $\mathcal{B}$ is an \textit{ordered tuple.} Thus, in order to apply the Equivalence Lemma to two sets of basis-values $(a_1, a_2, a_3, a_4, a_5, a_6)$ and $(b_1, b_2, b_3, b_4, b_5, b_6)$, we must ensure that $a_i = b_i$ holds for all $1\le i\le 6$.

\subsection{Basis Classification}

In this section, we complete the proof of the Orbit Space Classification of Theorem 2.2. The proof is divided into two lemmas which establish the equivalence of the matrices in each of the sets $\mathcal{A}_{p,q}^+$ and $\mathcal{A}_{p,q}^-$.

\begin{lemma}

Let $p > 0$ and $q \in \R$. Then the elements of the set
\[\mathcal{A}_{p,q}^{+} = \{ A \in S(4,\R) \mid A \not = \pm\sqrt{p} J, \operatorname{Pf}(A) = p, \operatorname{s}(A) = q\}\]

are equivalent under the group action $\rho$.

\end{lemma}

\begin{proof}

Let $A \in \mathcal{A}_{p,q}^+$ be given as
\[
A =
\begin{bmatrix}
0 & a & b & c \\
-a & 0 & d & e \\
-b & -d & 0 & f \\
-c & -e & -f & 0 \\
\end{bmatrix}.
\]

We shall split into cases in order to show that all matrices in $\mathcal{A}_{p,q}^+$ are equivalent.

\textbf{Case 1:} $b \not = 0$. In this case, the basis $\mathcal{B}_1 = \{v_1, v_2, v_3, v_4\}$, where
\[v_1 = 
\begin{bmatrix}
1 \\
0 \\
0 \\
0
\end{bmatrix}\]

\[v_2 = 
\begin{bmatrix}
0 \\
0 \\
\frac{1}{b} \\
0
\end{bmatrix}\]

\[v_3 = 
\begin{bmatrix}
-\frac{d}{b} \\
1 \\
-\frac{a}{b} \\
0
\end{bmatrix}\]

\[v_4 = 
\begin{bmatrix}
-\frac{f}{af-be+cd} \\
0 \\
\frac{c}{af-be+cd} \\
-\frac{b}{af-be+cd}
\end{bmatrix},\]

is a symplectic basis for $A$. Furthermore, the set of basis-values of $\mathcal{B}_1$ is 
\[\left(0, 1, 0, 0, -\frac{1}{af-be+cd}, \frac{a+f}{af-be+cd}\right).\]

As $af - be + cd = p$ and $a+f = q$ are fixed, the Equivalence Lemma implies that all matrices in Case 1 are equivalent.

\textbf{Case 2:} $b = 0$ and $c \neq 0$. In this case, the basis $\mathcal{B}_2 = \{v_1, v_2, v_3, v_4\}$, where
\[v_1 = 
\begin{bmatrix}
1 \\
0 \\
0 \\
0
\end{bmatrix}\]

\[v_2 = 
\begin{bmatrix}
0 \\
0 \\
0 \\
\frac{1}{c}
\end{bmatrix}\]

\[v_3 = 
\begin{bmatrix}
-\frac{e}{c} \\
1 \\
0 \\
-\frac{a}{c}
\end{bmatrix}\]

\[v_4 = 
\begin{bmatrix}
-\frac{f}{af+cd} \\
0 \\
\frac{c}{af+cd} \\
-\frac{b}{af+cd}
\end{bmatrix},\]

is a symplectic basis for $A$. Furthermore, the set of basis-values of $\mathcal{B}_2$ is
\[\left(0, 1, 0, 0, -\frac{1}{af+cd}, \frac{a+f}{af+cd}\right).\]

As $af + cd = p$ and $a + f = q$, the Equivalence Lemma implies that all matrices in Case 2 are equivalent. Furthermore, since the set of basis-values of the matrices in Case 2 is equal to the set of basis-values of the matrices in Case 1, the Equivalence Lemma implies that all matrices in Cases 1 and 2 are equivalent.

\textbf{Case 3:} $b = c = 0$ and $d \neq 0$. In this case, we may apply the symplectic permutation matrix $\tilde{P} \in \Sp(4)$ given by
\[
\tilde{P} =
\begin{bmatrix}
0 & 0 & 1 & 0 \\
0 & 0 & 0 & 1 \\
1 & 0 & 0 & 0 \\
0 & 1 & 0 & 0
\end{bmatrix}
\]

to yield the equivalent matrix $\tilde{P}^T A \tilde{P}$ in Case 2 with a nonzero $(1,4)$-entry. Hence, the matrices of Case 3 are equivalent to the matrices of Cases 1 and 2.

\textbf{Case 4:} $b = c = d = 0$ and $e \neq 0$. As $af = p \neq 0$, it follows that $a \neq 0$ and $f \neq 0$. Hence, we may apply the matrix $P_1 \in \Sp(4)$ given by
\[
P_1 = 
\begin{bmatrix}
0 & \frac{e}{a} & 0 & 0 \\
-\frac{a}{e} & 0 & 0 & \frac{a}{e} \\
0 & 0 & 0 & 1 \\
0 & 1 & -1 & 0 \\
\end{bmatrix}
\]

to yield the equivalent matrix
\[
P_1^T A P_1 =
\begin{bmatrix}
0 & 0 & a & 0 \\
0 & 0 & 0 & -f \\
-a & 0 & 0 & a+f \\
0 & f & -a-f & 0 \\
\end{bmatrix}
\]

with a nonzero $(1,3)$-entry. It follows that the matrices of Case 4 are equivalent to those of Cases 1, 2, and 3.

\textbf{Case 5:} $b = c = d = e = 0$. In this case, as $af = p$, the matrix $A$ is of the form
\[
A =
\begin{bmatrix}
0 & a & 0 & 0 \\
-a & 0 & 0 & 0 \\
0 & 0 & 0 & \frac{p}{a} \\
0 & 0 & -\frac{p}{a} & 0 \\
\end{bmatrix},
\]

with $a \neq \pm \sqrt{p}$. Hence, if $t = \frac{a}{\sqrt{p}}$, we may apply the matrix $P_2 \in \Sp(4)$ given by
\[
P_2 =
\begin{bmatrix}
1 & 0 & 0 & 1 \\
0 & \frac{1}{1-t^2} & \frac{t^2}{1-t^2} & 0 \\
0 & \frac{1}{1-t^2} & \frac{1}{1-t^2} & 0 \\
t^2 & 1 & 1 & 1 \\
\end{bmatrix}
\]

to yield the equivalent matrix
\[
P_2^T A P_2 =
\sqrt{p} \cdot
\begin{bmatrix}
0 & 0 & -t & 0 \\
0 & 0 & 0 & \frac{1}{t} \\
a & 0 & 0 & t + \frac{1}{t} \\
0 & -\frac{1}{t} & -t-\frac{1}{t} & 0 \\
\end{bmatrix},
\]

which has a nonzero $(1,3)$-entry. It follows that all matrices in Case 5 are equivalent to those in Case 1. Therefore, all matrices in $\mathcal{A}_{p,q}^+$ are equivalent, and we may conclude.

\end{proof}

\begin{lemma}

Let $p < 0$ and $q \in \R$. Then the elements of the set
\[\mathcal{A}_{p,q}^- = \{ A \in S(4,\R) \mid \operatorname{Pf}(A) = p, \operatorname{s}(A) = q\}\]

are equivalent under the group action $\rho$.

\end{lemma}

\begin{proof}

Let $A \in \mathcal{A}_{p,q}^-$ be given as 
\[
A =
\begin{bmatrix}
0 & a & b & c \\
-a & 0 & d & e \\
-b & -d & 0 & f \\
-c & -e & -f & 0 \\
\end{bmatrix}.
\]

By the analysis in the proof of Lemma 4.3, all matrices $A \in \mathcal{A}_{p,q}^-$ with at least one of $b, c, d, e$ nonzero are equivalent. It suffices to consider matrices $A$ with $b = c = d = e = 0$.

In this case, as $af = p$, the matrix $A$ is of the form
\[
A =
\begin{bmatrix}
0 & a & 0 & 0 \\
-a & 0 & 0 & 0 \\
0 & 0 & 0 & \frac{p}{a} \\
0 & 0 & -\frac{p}{a} & 0 \\
\end{bmatrix}.
\]

Hence, if $t = \frac{a}{\sqrt{-p}}$ we may apply the matrix $P \in \Sp(4)$ given by
\[
P =
\begin{bmatrix}
1 & 0 & 0 & t^2+1 \\
0 & \frac{1}{t^2+1} & -\frac{t^2}{(t^2+1)^2} & 0 \\
0 & 1 & \frac{1}{t^2+1} & 0 \\
-\frac{t^2}{t^2+1} & 1 & \frac{1}{t^2+1} & 1 \\
\end{bmatrix}
\]

to yield the equivalent matrix
\[
P^T A P =
\sqrt{-p} \cdot 
\begin{bmatrix}
0 & 0 & -\frac{t}{t^2+1} & 0 \\
0 & 0 & 0 & -t-\frac{1}{t} \\
\frac{t}{t^2+1} & 0 & 0 & t-\frac{1}{t} \\
0 & t+\frac{1}{t} & -t+\frac{1}{t} & 0 \\
\end{bmatrix},
\]

which has a nonzero $(1,3)$-entry. It follows that all matrices in $\mathcal{A}_{p,q}^-$ are equivalent, as desired.

\end{proof}

As a corollary of Theorem 2.1, the orbits $\mathcal{A}_{p,q}^+$ and $\mathcal{A}_{p,q}^-$ are distinct for all $p, q$. Hence, by Lemmas 4.3 and 4.4, we have the Orbit Space Classification of Theorem 2.2.

\section{Geometry of the Orbits in $\R^4$}

In this section, we investigate the shapes of the individual orbits of $S(4,\R) / \Sp(4)$ when considered as subsets of $\R^6$. We shall prove Theorem 2.3.

Let $p \in \R^*$, $q \in \R$. Then the orbits $\mathcal{J}_p^+$ and $\mathcal{J}_p^-$ are singletons. The orbit $\mathcal{A}_{p,q}$, on the other hand, is homeomorphic to the set
\[\{(a, b, c, d, e, f) \in \R^6 \mid af-be+cd = p, a+f = q\}.\]

Let the discriminant $\Delta = \frac{q^2}{4} - p$, and consider the set
\[D_{\Delta} = \{(b,c,d,e) \in \R^4 \mid be - cd \le \Delta \}.\]

For each point $(a,b,c,d,e,f) \in \mathcal{A}_{p,q}$, $(b,c,d,e)$ is constrained to $D_{\Delta}$. Furthermore, given $(b,c,d,e) \in D_{\Delta}$, there are at most two corresponding values of $(a, f)$, namely
\[(a,f) = \left(\frac{q}{2} + \sqrt{\Delta - (be-cd)}, \frac{q}{2} - \sqrt{\Delta - (be-cd)}\right)\]

and
\[(a,f) = \left(\frac{q}{2} - \sqrt{\Delta - (be-cd)}, \frac{q}{2} + \sqrt{\Delta - (be-cd)}\right).\]

If $(b,c,d,e)$ lies on the boundary of $D_{\Delta}$, these two solutions are identical. If $(b,c,d,e)$ lies in the interior of $D_{\Delta}$, these two pairs are distinct. Hence, $\mathcal{A}_{p,q}$ is homeomorphic to the \textit{double} of $D_{\Delta}$: two copies of $D_{\Delta}$ glued along their boundaries.

In order to determine the shapes of the orbits $\mathcal{A}_{p,q}$ for $\Delta > 0$ and $\Delta < 0$, we split into two cases.

\textbf{Case 1:} $\Delta > 0$.

Let $D^2$ denote the closed unit disk. Consider the homeomorphism $f : D^2 \times \R^2 \rightarrow D_{\Delta}$ defined as
\[
f(x,y,z,w) = (x\sqrt{z^2+w^2+\Delta}+z, w+y\sqrt{z^2+w^2+\Delta}, w-y\sqrt{z^2+w^2+\Delta}, x\sqrt{z^2+w^2+\Delta}-z).
\]

It follows that $D_{\Delta}$ is homeomorphic to $D^2 \times \R^2$. Furthermore, the double of $D^2$ is the sphere $\mathbb{S}^2$. Therefore, in this case, $\mathcal{A}_{p,q}$ is homeomorphic to $\mathbb{S}^2 \times \R^2$, as desired.

\textbf{Case 2:} $\Delta < 0$.

Consider the homeomorphism $g : \R^2 \times (D^2 \backslash \{0,0\}) \rightarrow D_{\Delta}$ defined as
\[
g(x,y,z,w) = \left(x + \frac{z\sqrt{x^2+y^2-\Delta}}{z^2+w^2}, \frac{w\sqrt{x^2+y^2-\Delta}}{z^2+w^2}+y, \frac{w\sqrt{x^2+y^2-\Delta}}{z^2+w^2}-y,
x - \frac{z\sqrt{x^2+y^2-\Delta}}{z^2+w^2}\right).
\]

It follows that $D_{\Delta}$ is homeomorphic to $\R^2 \times (D^2 \backslash \{0,0\})$. The double of $D^2 \backslash \{0,0\}$ is homeomorphic to $\R^2 \backslash \{0,0\}$, and it follows that in this case, $\mathcal{A}_{p,q}$ is homeomorphic to $\R^2 \times (\R^2 \backslash \{0,0\})$. The result follows.

Cases 1, and 2 yield Theorem 2.3.

\section{Nonlinear Case in $\R^4$}

\begin{Proposition}
    Given $\Omega_1$ and $\Omega_2$ matrix representations for symplectic form $\omega_1$ and $\omega_2$ respectively, suppose that there is a $\phi \in \Symp(\mathbb{R}^{2n},\omegastd)$ which means that $\phi^*\omega_1=\omega_2$, then would the invariant $s_i,1\le i \le n$ match up for matrices $\Omega_1(\phi(x))$ and $\Omega_2(x)$.
\end{Proposition}
\begin{proof}
    By the equation $\phi^*\omega_1=\omega_2$, we get to know that $P^T \Omega_1(\varphi(x)) P=\Omega_2(x)$, where $P = d\varphi_x$ denotes the linearization of $\varphi$ at $x$. Then we can conclude that $\Omega_1(\varphi(x))$ and $\Omega_2(x)$ are equivalent under the group action $rho$. Consequently by theorem 2.1, all the invariant will match up for two matrices $\Omega_1(\phi(x))$ and $\Omega_2(x)$. 
\end{proof}

In order to prove the global invariants of Theorem 2.6, we must restrict our attention to the symplectic forms
\[\omega = f_1(x_1, y_1) \mathrm{d} x_1 \wedge \mathrm{d} y_1 + f_2(x_2, y_2) \mathrm{d} x_2 \wedge \mathrm{d} y_2\]

and
\[\omega' = g_1(x_1, y_1) \mathrm{d} x_1 \wedge \mathrm{d} y_1 + g_2(x_2, y_2) \mathrm{d} x_2 \wedge \mathrm{d} y_2,\]

where $f_1, f_2, g_1$, and $g_2$ are nowhere vanishing.

We will show that if $\omega$ and $\omega'$ are $\omegastd$-symplectomorphic, then
\[\inf\{m_{\omega}(x) \mid x \in \R^4\}=\inf\{m_{\omega'}(x) \mid x \in \R^4\}.\]

The three other global invariants are handled similarly.

We will need the following lemma, a corollary of Theorem 2.4.

\begin{lemma}

Consider the symplectic forms
\[\omega_1 = f_1(x_1, y_1) \mathrm{d} x_1 \wedge \mathrm{d} y_1 + f_2(x_2, y_2) \mathrm{d} x_2 \wedge \mathrm{d} y_2\]

and
\[\omega_2 = g_1(x_1, y_1) \mathrm{d} x_1 \wedge \mathrm{d} y_1 + g_2(x_2, y_2) \mathrm{d} x_2 \wedge \mathrm{d} y_2,\]

where $f_1, f_2, g_1,$ and $g_2$ are nowhere vanishing. Let $\varphi \in \Symp(\R^4, \omegastd)$. If $\varphi^* \omega_1 = \omega_2$, then for each $x \in \R^4$, either:

\begin{itemize}

    \item $f_1(\varphi^1, \varphi^2) = g_1(x_1, y_1)$ and $f_2(\varphi^3, \varphi^4) = g_2(x_2, y_2)$, or
    
    \item $f_1(\varphi^1, \varphi^2) = g_2(x_2, y_2)$ and $f_2(\varphi^3, \varphi^4) = g_1(x_1, y_1)$.
    
\end{itemize}

(Here $\varphi^i : \R^4 \rightarrow \R$ are the component functions of $\varphi = (\varphi^1, \varphi^2, \varphi^3, \varphi^4)$.)

\end{lemma}

\begin{proof}

By Theorem 2.4, the Pfaffian and sum functions of the matrices
\[\Omega_1(\varphi(x)) =
\begin{bmatrix}
0 & f_1(\varphi^1(x), \varphi^2(x)) & 0 & 0 \\
-f_1(\varphi^1(x), \varphi^2(x)) & 0 & 0 & 0 \\ 
0 & 0 & 0 & f_2(\varphi^3(x), \varphi^4(x)) \\
0 & 0 & -f_2(\varphi^3(x), \varphi^4(x)) & 0
\end{bmatrix}
\]

and
\[\Omega_2(x) =
\begin{bmatrix}
0 & g_1(x_1, y_1) & 0 & 0 \\
-g_1(x_1, y_1) & 0 & 0 & 0 \\ 
0 & 0 & 0 & g_2(x_2, y_2) \\
0 & 0 & -g_2(x_2, y_2) & 0
\end{bmatrix}
\]

must match at every point $x \in \R^4$.

Hence, for each $x \in \R^4$,
\[f_1(\varphi^1(x), \varphi^2(x))f_2(\varphi^3(x), \varphi^4(x)) = g_1(x_1, y_1)g_2(x_2, y_2)\]
and
\[f_1(\varphi^1(x), \varphi^2(x)) + f_2(\varphi^3(x), \varphi^4(x)) = g_1(x_1, y_1) + g_2(x_2, y_2).\]

The result follows.

\end{proof}

By Lemma 6.1, $\{f_1(\varphi^1(x), \varphi^2(x)), f_2(\varphi^3(x), \varphi^4(x))\} = \{g_1(x_1, y_1), g_2(x_2, y_2)\}$ for all $x \in \R^4$, and hence, $m_{\omega}(\varphi(x)) = m_{\omega'}(x)$ for all $x \in \R^4$. As $\varphi$ is a bijection, it follows that the sets $\{m_{\omega}(x) \mid x \in \R^4\}$ and $\{m_{\omega'}(x) \mid x \in \R^4\}$ are equal. Therefore,
\[\inf\{m_{\omega}(x) \mid x \in \R^4\}=\inf\{m_{\omega'}(x) \mid x \in \R^4\},\]

and Theorem 2.5 follows.

Lastly, we consider the following corollary of the Curry-Pelayo-Tang Stability Theorem \cite[Main Theorem]{Curry2018}.

\begin{theorem}[{\cite[Example 4.1]{Curry2018}}]

Consider $\R^4$ with coordinates $(x_1, y_1, x_2, y_2)$. Let $U$ be an open subset of $\R^4$, and let $f_1, f_2, f_3, f_4 \in C^{\infty}(U)$. Note that
\[
\omega = f_1 \mathrm{d} x_1 \wedge \mathrm{d} y_1 + f_2 \mathrm{d} x_2 \wedge \mathrm{d} y_2
\]
is a symplectic form if and only if each of the $f_i$ is nowhere vanishing and depends only on the coordinates $x_i, y_i$.

Define the symplectic forms
\[
\omega_t = f_1(t,x_1,y_1) \mathrm{d} x_1 \wedge \mathrm{d} y_1 + f_2(t,x_2,y_2) \mathrm{d} x_2 \wedge \mathrm{d} y_2, t \in [0,1],
\]
and suppose that the $f_i$ are bounded away from zero with bounded time derivative. Then there exists a smooth path of $\varphi_t \in \Diff(\R^4)$, $t \in [0,1]$, such that $\varphi_t^* \omega_t = \omega_0$.
  
\end{theorem}

As mentioned in Section 2.2, Theorem 6.2 motivates the following question.

\begin{question}
    Under the conditions of Theorem 6.1, does there exist a smooth path of $\varphi_t \in \Symp(\R^4, \omegastd)$, $t \in [0, 1]$, such that $\varphi_t^* \omega_t = \omega_0$?
\end{question}

As we shall see in the Example below, this is not necessarily the case. We use Lemma 6.1 to provide an infinite collection of smooth families of symplectic forms $\omega_t$, $t \in [0,1]$, which are diffeomorphic but \textit{not} $\omegastd$-symplectomorphic.

\begin{example}

Fix $t_1, t_2 \in [0,1]$, and assume without loss of generality that $t_1 > t_2$. Let $g : \R^2 \rightarrow \R$ be a function satisfying $\inf \{g(x,y) \mid x, y \in \R\} = c > 0$. Let $f_1(t, x_1, y_1) = g(x_1, y_1) + t$ and let $f_2(t, x_2, y_2) = g(x_2, y_2) + t$. Consider the family of symplectic forms 
\[\omega_t = f_1(t,x_1,y_1) \mathrm{d} x_1 \wedge \mathrm{d} y_1 + f_2(t,x_2,y_2) \mathrm{d} x_2 \wedge \mathrm{d} y_2, t \in [0,1].\]

Then $\omega_{t_1}$ and $\omega_{t_2}$ are diffeomorphic but not $\omegastd$-symplectomorphic.

To see that $\omega_{t_1}$ and $\omega_{t_2}$ are diffeomorphic, apply Theorem 6.2. To see that $\omega_{t_1}$ and $\omega_{t_2}$ are not $\omegastd$-symplectomorphic, note that by Lemma 6.1, if $\varphi^* \omega_{t_1} = \omega_{t_2}$ for some $\varphi \in \Symp(\R^4,\omegastd)$, then for each $x \in \R^4$, either $f_1(t_1, \varphi^1(x), \varphi^2(x)) = f_1(t_2, x_1, y_1)$ or $f_1(t_1, \varphi^1(x), \varphi^2(x)) = f_2(t_2, x_2, y_2)$. Select $x = (u, v, u, v)$ such that $g(u, v) < c + (t_1 - t_2)$. In either of the two cases, we find $g(\varphi^1(x), \varphi^2(x)) + (t_1 - t_2) = g(u, v)$, which is a contradiction to $g(u,v) < c + (t_2 - t_1)$. It follows that $\omega_{t_1}$ and $\omega_{t_2}$ are not $\omegastd$-symplectomorphic, and therefore, the symplectic forms $\omega_t$, $t \in [0,1]$ are diffeomorphic but not $\omegastd$-symplectomorphic.

In particular, consider $g(x,y) = x^2 + y^2 + 1$, for which
\[\omega_t = (x_1^2+y_1^2+t+1)\mathrm{d} x_1 \wedge \mathrm{d} y_1 + (x_2^2+y_2^2+t+1)\mathrm{d} x_2 \wedge \mathrm{d} y_2, t \in [0,1]\]

From the above, it follows that the symplectic forms $\omega_t$, $t \in [0,1]$ are diffeomorphic but not $\omegastd$-symplectomorphic.

\end{example}

\section{Higher Dimensional Linear Case}

\subsection{Equivalence Lemma for $\R^{2n}$}

In this section, we state the Equivalence Lemma in its full generality. First, we must extend the notion of a symplectic basis as introduced in Section 4.1.

\begin{definition}[Symplectic basis]

Let $\omega$ be a linear symplectic form on $\R^{2n}$, and let $\mathcal{B} = \{v_1, v_2, \ldots, v_{2n}\}$ be a basis for $(\R^{2n}, \omega)$. Then $\mathcal{B}$ is said to be a \textit{symplectic basis for $(\R^{2n}, \omega)$} if the following equalities hold:

\begin{itemize}

    \item $\omega(v_{2i-1}, v_{2i}) = 1$ for $1\le i\le n$; and,
    
    \item $\omega(v_i, v_j) = 0$ for $i < j$ and $(i, j) \neq (2k-1, 2k)$ for all $1\le k\le n$.

\end{itemize}

\end{definition}

As in Theorem 4.1, the Modified Gram-Schmidt Algorithm may be used to prove the the following result \cite[Theorem 2.10.4]{FirstSteps}.

\begin{theorem}

Every symplectic vector space $(\R^{2n}, \omega)$ has a symplectic basis.

\end{theorem}

We continue with a generalization of Lemma 4.2.

\begin{lemma}

Let $\{v_i\}_{i=1}^{2n}$ and $\{w_i\}_{i=1}^{2n}$ be two bases for $\R^{2n}$. Define the linear transformation $P: \R^{2n}\rightarrow \R^{2n}$ such that $Pv_i = w_i$ for all $1\le i\le 2n$. Then $P^T J P = J$ if and only if $\omegalstd(v_i, v_j) = \omegalstd(w_i, w_j)$ for all $1\le i\le j\le 2n$.

\end{lemma}

\begin{proof}

($\rightarrow$) If $P^T J P = J$, then 
\[w_i^T J w_j = (Pv_i)^T J (Pv_j) = v_i^T (P^T J P) v_j = v_i^T J v_j,\]
as desired.

($\leftarrow$) Follows from Definition 2.5.

\end{proof}

\begin{definition}[Basis-values]

Given a symplectic basis $\mathcal{B} = \{v_1, v_2, \ldots, v_{2n}\}$ for the linear symplectic form $A$ on $\R^{2n}$, the \textit{set of basis-values of $\mathcal{B}$} is the ordered ${n}\choose{2}$-tuple \[(\omegalstd(v_1, v_2), \omegalstd(v_1, v_3), \omegalstd(v_1, v_4),\ldots, \omegalstd(v_{2n-1}, v_{2n})).\]

\end{definition}

We are ready to state the Equivalence Lemma in $\R^{2n}$.

\begin{corollary}[Equivalence Lemma in $\R^{2n}$]

Let $A, B \in S(2n,\R)$. Let $\mathcal{B}_1 = \{v_1, v_2, \ldots, v_{2n}\}$ be a symplectic basis for $A$ and let $\mathcal{B}_2 = \{w_1, w_2, \ldots, w_{2n}\}$ be a symplectic basis for $B$. If the set of basis-values of $\mathcal{B}_1$ is equal to the set of basis-values of $\mathcal{B}_2$, then $A$ and $B$ are equivalent under the group action $\rho$.

\end{corollary}

\begin{proof}

Let $P : \R^{2n} \rightarrow \R^{2n}$ be the linear transformation such that $Pv_i = w_i$ for $1\le i\le 2n$. As the set of basis-values of $\mathcal{B}_1$ is equal to the set of basis-values for $\mathcal{B}_2$, by Lemma 7.2, it follows that $P^T J P = J$.

As $\mathcal{B}_1$ and $\mathcal{B}_2$ are symplectic bases for $A$ and $B$, respectively, $v_i^T A v_j = w_i^T B w_j$ for all $1\le i,j \le 2n$. Hence, $v_i^T A v_j = (Pv_i)^T B (Pv_j)^T$ for all $1\le i,j \le 2n$, which implies that $v_i^T A v_j = v_i^T (P^T B P) v_j$ for all $1\le i,j \le 2n$. 

As $\{v_1, v_2, \ldots, v_{2n}\}$ is a basis for $\R^{2n}$, it follows that $x^T A y = x^T (P^T B P) y$ for all $x, y \in \R^{2n}$. In particular, for the standard basis vectors $x = e_i$ and $y = e_j$, we have $A_{i,j} = e_i^T A e_j = e_i^T (P^T B P) e_j = (P^T B P)_{i,j}$. It follows that all of the entries of $A$ and $P^T B P$ are equal, so that $A = P^T B P$. The result follows.

\end{proof}

\subsection{Proof of Invariants $s_k$ in $\R^{2n}$}

\begin{definition}
   In this section, we derive the invariants $s_k$ Theorem 2.8. The formal definition of $s_k$ is stated as below:

Let $A \in S(2n,\mathbb{R})$, and let $0\le k\le n-1$. For a subset $S \subset \{1,2,...,n\}$ with $|S| = k$, define the $(2n-2k) \times (2n-2k)$ matrix $A_{S}$ as the submatrix of $A$ formed by removing rows and columns $2i-1$ and rows and columns $2i$ for all $i \in S$.Let $s_k(A)$ denote the sum
\[
\sum_{\substack{S \subseteq \{1, 2, \ldots, n\} \\ |S|=k}} \operatorname{Pf}(A_{S}).
\]
Then for any $P \in \Sp(2n)$, $s_k(A) = s_k(P^T A P)$. In other words, $s_k$ is invariant under the group action $\rho$.
\end{definition}
At the same time, we can notice that the pfaffian and sum function we discussed in the section 3 of this paper are exactly $s_0$ and $s_{n-1}$ here respectively, which indicates that our result here is a generalized version of the section 3 proof. 
\begin{definition}[Pfaffian of a $2n \times 2n$ matrix]

The \textit{Pfaffian} of a $2n \times 2n$ matrix $A$ is the quantity
\[\operatorname{Pf}(A) = \sum_{\sigma} \text{sgn}(\sigma) \prod_{i=1}^{n} A_{\sigma(2i-1), \sigma(2i)},\]

where $\text{sgn}(\sigma)$ denotes the signature of $\sigma$ and the summation is taken over all permutations $\sigma$ of $\{1, 2, \ldots, 2n\}$. 
(Here, $A_{i, j}$ denotes the $(i, j)$ entry of $A$.)

\end{definition}

According to the above permutation definition of pfaffian, we can expand $\operatorname{Pf}{(tJ+A)}$ for $A\in S(2n,\mathbb{R})$ into a polynomial $g(t)$ with $\deg(g)=n$, such that
\begin{equation}
    \operatorname{Pf}{(tJ+A)}=t^n+t^{n-1}\sigma_{n-1}(A)+t^{n-2}\sigma_{n-2}(A)+\cdot\cdot\cdot+\sigma_0(A)
\end{equation}
where $\sigma_k$ is the coefficient of $t^k$.
\begin{definition}
Given a two form $\omega$ associated to any skew-symmetric $2n\times2n$ matrice $A={a_{ij}}$, such that 
\begin{equation*}
\omega=\sum_{i<j} a_{ij}e_i\wedge e_j.  
\end{equation*}
where $\{e_1, e_2,...,e_{2n}\}$ is the standard basis of $\mathbb{R}^2n$. The pfaffian of matrix $A$ is then defined as
\begin{equation}
    \frac{1}{n!}\omega^n=\operatorname{Pf}(A)e_1\wedge e_2\wedge\cdot\cdot\cdot \wedge e_{2n}
\end{equation}
\end{definition}
Accordingly, we can use the above definition of pfaffian in terms of 2-forms to prove the following result.
\begin{Proposition}
$\sigma_k$ is invariant under the group action $\rho$ which means that $\sigma_k(P^TAP)=\sigma_k(A)$ for any $A\in S(2n,R)$ and $P\in \Sp(2n)$.
\end{Proposition}

\begin{proof}
According to lemma 7.3 and 7.4, we have $\operatorname{Pf}(tJ +A) = \operatorname{Pf}(P^T (tJ +A)P)$ for any $P\in \Sp(2n)$. Consequently, $P^TJP=J$, and then

\begin{align}
    \operatorname{Pf}(tJ+A)
    &=\operatorname{Pf}(P^T(tJ+A)P)\\
    &=\operatorname{Pf}(tJ+P^TAP)\\
    &=t^n+t^{n-1}\sigma_{n-1}(P^TAP)+t^{n-2}\sigma_{n-2}(P^TAP)+\cdot\cdot\cdot+\sigma_0(P^TAP)
\end{align}
Comparing the coefficient of formula (1) to formula (4), we obtain that $\sigma_k(A)=\sigma_k(P^TAP)$, for $k=0,... n-1$. 
\end{proof}
Once we can prove that $\sigma_k=s_k$, we are done with the above proposition on invariant. 
Let $\mathbf{N}=\{1,2,...,2n\}$, for any $A=\{a_{ij}\}_{1\le i,j \le 2n}\in S(2n, \mathbb{R})$, we define the associated 2-form $\omega$ such that
\[
\omega=\sum_{\substack{i<j \\ {i,j}\in \mathbf{N}}} a_{ij}dx_i\wedge dx_j.
\]
where the set $\{dx_1, dx_2,..., dx_{2n}\}$ is the standard basis for the vector space $(T_p(\mathbb{R}^{2n}))^*$. Then the pfaffian of A is determined by the following identity.
\begin{equation}
  \frac{1}{n!}\omega^n=\operatorname{Pf}(A)dx_1\wedge dx_2\wedge\cdot\cdot\cdot\wedge dx_{2n}
\end{equation}
where $\omega^n=\omega\wedge\cdot\cdot\cdot\wedge \omega$.

\begin{Proposition}Given coefficient $\sigma_k$ of term $t^k$ in the polynomial expansion of $\operatorname{Pf}(tJ+A)$, standard symplectic form $\omega_0$ and 2-form $\omega$ associated with $A$, we have
\begin{equation}
    \sigma_kdx_1\wedge dx_2\wedge\cdot\cdot\cdot\wedge dx_{2n}=\frac{1}{k!(n-k)!}\omega_0^k\wedge \omega^{n-k}
\end{equation}
\end{Proposition}

\begin{proof}
For the standard symplectic form $\omega_0$ associated with the matrix representation $J\in S(2n, \mathbb{R})$, we have 
\begin{equation*}
  \omega_0=\sum_{i=1}^{n}dx_{2i-1}\wedge dx_{2i} 
\end{equation*}
   Now we consider the matrix $tJ+A$ which has the associated 2-form $t\omega_0+\omega$. By formula (6), we then have 
\begin{equation*}
   \frac{1}{n!}(t\omega_0+\omega)^n=\operatorname{Pf}(tJ+A)dx_1\wedge dx_2\wedge\cdot\cdot\cdot\wedge dx_{2n}   
\end{equation*}
   Using the binomial theorem, we can then expand the $\frac{1}{n!}(t\omega_0+\omega)^n$ and note that 
\begin{equation}
    \frac{1}{n!}(t\omega_0+\omega)^n=\sum_{k=0}^{n}\frac{1}{k!(n-k)!}t^k\omega_0^k\wedge \omega^{n-k}
\end{equation}
    Comparing the formula (8) with formula  (5), we can then conclude that 
\begin{equation}
    \sigma_kdx_1\wedge dx_2\wedge\cdot\cdot\cdot\wedge dx_{2n}=\frac{1}{k!(n-k)!}\omega_0^k\wedge \omega^{n-k}
\end{equation}
\end{proof}
Therefore, the equivalence of $\sigma_k(A)=s_k(A)$ is given by the following Proposition.
\begin{Proposition}
     For $k=0,1$, we have 
     \begin{equation}
         \frac{1}{k!(n-k)!}\omega_0^k\wedge \omega^{n-k}=s_k dx_1\wedge dx_2\wedge\cdot\cdot\cdot\wedge dx_{2n}
     \end{equation}
\end{Proposition}
\begin{proof}
For $k=0$, formula (10) holds trivially which is just the formula (6).\\
For $k=1$, we have 
\begin{equation}
    \frac{1}{(n-1)!}\omega_0^k\wedge \omega^{n-1}=\frac{1}{(n-1)!}\sum_{i=1}^{n}dx_{2i-1}\wedge dx_{2i}\wedge \omega^{n-1} 
\end{equation}
For each term $dx_{2i-1}\wedge dx_{2i}\wedge \omega^{n-1}$, we note that, only those terms which do not contain $dx_{2i-1}$ or $dx_{2i}$ in $\omega^{n-1}$ term will contribute to the term $dx_{2i-1}\wedge dx_{2i}\wedge \omega^{n-1}$.By this observation, if we construct a new 2-form $\omega_{\hat{i}}$, such that
\begin{equation}
    \omega_{\hat{i}}=\sum_{\substack{k<l \\ {k,l} \in\mathbf{N}\setminus \{2i-1,2i\}}}a_{kl}dx_k\wedge dx_l
\end{equation}
Consequently we would have
\begin{equation*}
    dx_{2i-1}\wedge dx_{2i}\wedge \omega^{n-1}= dx_{2i-1}\wedge dx_{2i}\wedge \omega_{\hat{i}}^{n-1}
\end{equation*}
Consequently by formula (6), we have the following:
\begin{equation}
   \frac{\omega_{\hat{i}}^{n-1}}{(n-1)!}=\operatorname{Pf}(A_{i}) dx_1\wedge dx_2\wedge\cdot\cdot\wedge \widehat{dx_{2i-1}}\wedge \widehat{dx_{2i}} \cdot\cdot\cdot\wedge dx_{2n}
\end{equation}
where the notation $\operatorname{Pf{i}}$ is the pfaffian of the submatrix $A_{S}$ formed by formed by removing rows and columns $2i-1$ and rows and columns $2i$ for all $i \in S$, where $S \subset \{1,2,...,n\}$ with $|S| = k$. In the case of $k=1$, $\operatorname{Pf}(A_{i})$ refer to the pfaffian of some submatrix $A_S$ formed by removing rows and columns $2i-1$ and rows and columns $2i$ for one particular $i \in S$ as $|S| =1$. Therefore, the formula $11$ gives
\begin{equation*}
     \frac{1}{(n-1)!}\omega_0^k\wedge \omega^{n-1}=\sum_{i=1}^n\operatorname{Pf}(A_{i}) \bigwedge_{i=1}^{2n}dx_i=s_1 \bigwedge_{i=1}^{2n}dx_i
\end{equation*}
\end{proof}
Inspired by the proof of invariant when $k=1$, we can derive the following Proposition directly.
\begin{Proposition}
     For $k=0,1...,n-1$, we have 
     \begin{equation}
         \frac{1}{k!(n-k)!}\omega_0^k\wedge \omega^{n-k}=s_k dx_1\wedge dx_2\wedge\cdot\cdot\cdot\wedge dx_{2n}
     \end{equation}
\end{Proposition}
\begin{proof}
Construct an increasing posive integer sequence $\{b_i\},i\in \{1,2,...,n\}$, then we have
\begin{align*}
   \frac{1}{k!(n-k)!}\omega_0^k\wedge \omega^{n-k}
   &=\frac{1}{k!(n-k)!}(\sum_{i=1}^{n}dx_{2i-1}\wedge dx_{2i})^k\wedge \omega^{n-k}\\
   &=\frac{1}{(n-k)!}\sum_{i=1}^{n}\bigwedge_{b_i}dx_{2b_i-1}\wedge dx_{2b_i}\wedge \omega^{n-k}\\
   &=\bigwedge_{b_i}dx_{2b_i-1}\wedge dx_{2b_i}\wedge \omega_{\hat{b_i}}^{n-k}.
\end{align*}
where $\omega_{\hat{b_i}}$ is constructed similar as in the case $k=1$, such that 
\begin{equation*}
    \omega_{\hat{b_i}}=\sum_{k_i<l_i}\bigwedge_{b_i}dx_{2b_i-1}\wedge dx_{2b_i}
\end{equation*}
where $\{k_i\}$ and $\{l_i\}$ are two increasing positive sequences such that $\{k_i,l_i\}\in \mathbf{N}\setminus \{2b_i-1,2b_i\}$. Using the similar techniques with $k=1$ case, we would have 
\begin{equation*}
     \frac{\omega_{\hat{a_i}}^{n-k}}{(n-k)!}=\operatorname{Pf}(A_{\{a_i\}}) dx_1\wedge dx_2\wedge\cdot\cdot \bigwedge_{b_i}\widehat{dx_{2b_i-1}\wedge dx_{2b_i}}\cdot\cdot\cdot\wedge dx_{2n}
\end{equation*}
      Consequently we will get 
\begin{equation*}
    \frac{1}{k!(n-k)!}\omega_0^k\wedge \omega^{n-k}=\operatorname{Pf}(A_{a_i}) \bigwedge_{i=1}^{2n}dx_i=s_k(A)\bigwedge_{i=1}^{2n}dx_i
\end{equation*}
\end{proof}
Consequently, we can show that $\sigma_k=s_k$ which are the invariant. According to the invariant proof of $\sigma_k$, we then get to know that the quantity $s_k$ are also invariants. Hence the theorem 2.8 is justified. 

\section{Higher dimensional nonlinear case}
\subsection{The proof of invariants}

\begin{Proposition}
     Consider symplectic forms 
     \begin{equation*}
         \omega_1=\sum_{i=1}^n f_i(x_i,y_i)dx_i\wedge dy_i
     \end{equation*}
and 
\begin{equation*}
         \omega_2=\sum_{i=1}^n g_i(x_i,y_i)dx_i\wedge dy_i
\end{equation*}
If $\phi \in \Symp(\mathbb{R}^{2n}, \omega_{std})$, such that $\phi^*\omega_1=\omega_2$, $\phi=(\phi^1,\phi^2,...,\phi^{2n})$
Then we would the following equality:
Then the set $\{f_i(\phi^{2i-1},\phi^{2i})\}$ and $\{g_j(x_i,y_i)\}$ for some $i,j$ such that $1<i,j<n$ at every point in $\mathbb{R}^{2n}$. 
\end{Proposition}

\begin{proof}
    Given $\Omega_1(x)$ and $\Omega_2(x)$ are matrix representation of $\omega_1$ and $\omega_2$ respectively, then we would have
\begin{equation}
  \operatorname{Pf}{(tJ+\Omega_1(\phi(x))}=\operatorname{Pf}{(tJ+\Omega_2(x))}
\end{equation}
Since both $\Omega_1(\phi(x))$ and $\Omega_2(x)$ are block diagonal matrices and according to our proof of $n$ invariant in $\mathbb{R}^{2n}$, we can conclude that the invariant $s_i(n)$= $(n-i)$th symmetric sum of zeros of $P(t)$, formulated as below:
\begin{equation}
 P(t)=\prod_{i=1}^n(t+f_i(\phi^{2i-1},\phi^{2i}))=\prod_{i=1}^n(t+g_i(x_i,y_i))
\end{equation}
for all $t\in \mathbb{R}^{2n}$. 
\end{proof}
Hence the theorem 2.5 is justified by the above proposition. In the similar way as $\R^4 cases$, for $x\in \mathbb{R}^{2n}$ we can define 
\[m_{\omega}(x) = \min\{f_i(x_i,y_i), 1\le i \le n\}\]
and 
\[M_{\omega}(x) = \max\{f_i(x_i,y_i), 1\le i \le n\}.\]
In other words, if $\omega_1$ is $\omegastd$-symplectomorphic to the symplectic form
\[\omega_2 =\sum_{i=1}^n g_i(x_i,y_i)dx_i\wedge dy_i\]

where $g_i$ are nowhere vanishing, then the four global invariants of $\omega_1$ and $\omega_2$ must match.

\begin{itemize}
    \item $\inf \{m_{\omega}(x) \mid x \in \mathbb{R}^{2n}\}$,
    
    \item $\sup\{m_{\omega}(x) \mid x \in \mathbb{R}^{2n}\}$,
    
    \item $\inf\{M_{\omega}(x) \mid x \in \mathbb{R}^{2n}\}$,
    
    \item $\sup\{M_{\omega}(x) \mid x \in \mathbb{R}^{2n}\}$.
\end{itemize}
Consequently, the theorem 2.6 is justified by the above natural constructions as well. 

\section{Prospectives}

In this section, we present conjectures and questions for further investigation. We shall begin with propsectives regarding the linear case. Our first conjecture stems from our observation that in $\R^4$, the Pfaffian and sum function invariants determine the Orbit Space Classification of $S(4,\R) / \Sp(4)$.

\begin{conjecture}

The $n$ invariants described in Conjecture 1 determine the orbit space $S(2n,\R) / \Sp(2n)$. In order words, if $A, B \in S(2n,\R)$ with $A \neq cJ$ for all $c \in \R$ and $B \neq cJ$ for all $c \in \R$ and if $\operatorname{s_k}(A) = \operatorname{s_k}(B)$ for all $0 \le k \le n-1$, then $A$ and $B$ are equivalent under the group action $\rho$.

\end{conjecture}

For example, Conjecture 2 implies that the orbit space $S(6,\R) / \Sp(4)$ consists of the following families of distinct orbits:

\begin{itemize}

    \item For all $p > 0$, 
    \[\mathcal{J}_p^{+} = \{\sqrt{p}J\}\]
    
    \item For all $p > 0$, 
    \[\mathcal{J}_p^{-} = \{-\sqrt{p}J\}\]
    
    \item For all $p > 0$, $q \in \R$, $r \in \R$,
    \[\mathcal{A}_{p,q,r}^{+} = \{ A \in S(6,\R) \mid A \not = \pm\sqrt{p} J, \operatorname{Pf}(A) = p, \operatorname{s_1}(A) = q, \operatorname{s_2}(A) = r\}\]
    
    \item For all $p < 0$, $q \in \R$, $r \in \R$,
    \[\mathcal{A}_{p,q,r}^{-} = \{ A \in S(6,\R) \mid \operatorname{Pf}(A) = p, \operatorname{s_1}(A) = q, \operatorname{s_2}(A) = r\}\]
    
\end{itemize}

We now turn to prospectives regarding the equivalence of nonlinear symplectic forms. The first asks if Question 6 can be salvaged.

\begin{question}
    
    Let $f_1, f_2$ be functions, and define the symplectic forms
    \[\omega_t = f_1(t, x_1, y_1) \mathrm{d} x_1 \wedge \mathrm{d} y_1 + f_2(t, x_2, y_2) \mathrm{d} x_2 \wedge \mathrm{d} y_2, t \in [0, 1].\]
    
    What is a sufficient condition on $f_1, f_2$ under which there exists a smooth path of $\varphi_t \in \Symp(\R^4, \omegastd)$, $t \in [0,1]$, such that $\varphi^* \omega_t = \omega_0$?
    
\end{question}

Finally, we present a question regarding nonlinear symplectic forms on $\R^{2n}$.

\begin{question}

    How do the results of Section 6 generalize to $\R^{2n}$? What do the proposed invariants of Conjectures 1 and 2 imply about nonlinear symplectic forms on $\R^{2n}$?

\end{question}

\pagebreak
\nocite{Lee}
\nocite{Qiaochu}

\bibliographystyle{plain}
\bibliography{references}

\begin{thebibliography}{10}

\bibitem{PfaffianProof}
Tracale Austin, Hans Bantilan, Isao Jonas, and Paul Kory.
\newblock The {P}faffian {T}ransformation, March 2007.

\bibitem{Curry2018}
Sean Curry, {\'A}lvaro Pelayo, and Xiudi Tang.
\newblock Symplectic stability on manifolds with cylindrical ends.
\newblock {\em The Journal of Geometric Analysis}, July 2018.

\bibitem{DefinitionsOfNonlinearForms}
Ana~Cannas da~Silva.
\newblock {\em Lectures on Symplectic Geometry}.
\newblock Lecture Notes in Mathematics. Springer-Verlag, 2006.

\bibitem{Darboux1882}
Jean-Gaston Darboux.
\newblock Sur le probl\`eme de {Pfaff}.
\newblock {\em Bulletin des Sciences Mathématiques et Astronomiques},
  6(1):14--36, 1882.

\bibitem{ElGr}
Yakov Eliashberg and Mikhail Gromov.
\newblock Convex {S}ymplectic {M}anifolds.
\newblock In {\em Several Complex Variables and Complex Geometry, Part II},
  volume~52 of {\em Proc. Sympos. Pure Math.}, pages 135--162. Amer. Math.
  Soc., Providence, RI, 1991.

\bibitem{Gr85}
Mikhail Gromov.
\newblock Pseudo holomorphic curves in symplectic manifolds.
\newblock {\em Invent. Math.}, 82(2):307--347, 1985.

\bibitem{Lee}
Jack Lee.
\newblock The {M}anifold of ${D} = \{x,y,z,w \mid xw-yz \le 0\}$.
\newblock Mathematics Stack Exchange.
\newblock https://math.stackexchange.com/q/2881943.

\bibitem{McD88}
Dusa McDuff.
\newblock The symplectic structure of {K}\"ahler manifolds of nonpositive
  curvature.
\newblock {\em J. Differential Geom.}, 28(3):467--475, 1988.

\bibitem{McD04}
Dusa McDuff.
\newblock {Lectures on Groups of Symplectomorphisms}.
\newblock {\em ArXiv e-prints}, April 2004.

\bibitem{FirstSteps}
Andrew McInerney.
\newblock {\em First Steps in Differential Geometry: Riemannian, Contact,
  Symplectic}.
\newblock Springer, 2013.

\bibitem{Ld01}
Leonid Polterovich.
\newblock {\em The geometry of the group of symplectic diffeomorphisms}.
\newblock Lectures in Mathematics. ETH Z\"urich. Birkh\"auser Verlag, Basel,
  2001.

\bibitem{Salamon}
Dietmar Salamon.
\newblock {Uniqueness of Symplectic Structures}.
\newblock {\em ArXiv e-prints}, November 2012.

\bibitem{Qiaochu}
Qiaochu Yuan.
\newblock {S}et defined by $xy-zw=1$.
\newblock Mathematics Stack Exchange.
\newblock https://math.stackexchange.com/q/160628.

\end{thebibliography}

\pagebreak
\end{document}